\documentclass[11pt]{amsart}

\usepackage[utf8x]{inputenc}
\usepackage[english]{babel}

\usepackage[pdftex]{graphicx}
\usepackage[margin=2.8cm]{geometry}

\usepackage{amsfonts}
\usepackage{amsmath}
\usepackage{amssymb}
\usepackage{amsthm}
\usepackage{dsfont}
\usepackage{tikz}
\usepackage{pgfplots}

\usepackage[T1]{fontenc}

\usepackage{paralist}

\usepackage[colorlinks,link color=blue, cite color=blue,pagebackref,pdftex]{hyperref}

\usepackage{etoolbox}
\patchcmd{\thmhead}{(#3)}{#3}{}{}

\renewcommand\emptyset{\varnothing}
\renewcommand\phi{\varphi}
               
\newcommand\Z{\mathbb{Z}}               
               
\newcommand\R{\mathbb{R}}

\newcommand{\halflozenge}{
    \begin{tikzpicture}
    \draw[semithick, scale=0.2] (-0.5, 0) -- (0, -0.75)-- (.5, 0)-- (0, .75)-- 
    (-.5, 0);
    \fill[thin, fill=black, scale=.2] (0, -.75) -- (.5, 0)-- (0, .75)-- (0, 
    -.75);
    \end{tikzpicture}
}

\newcommand\half[2]{\mathbb{H}_{#1}{#2}}

\newcommand\Pol[1]{\mathcal{P}(#1)}

\newcommand\ehrval{\varepsilon}

\newcommand\id{\mathsf{id}}
\newcommand\Ehr{\mathsf{Ehr}}
\newcommand\ehr{\mathsf{ehr}}
\newcommand\des{\mathsf{des}}
\newcommand\Des{\mathsf{Des}}

\newcommand\floor[1]{\left\lfloor {#1} \right\rfloor}

\def\x{\mathbf x}
\def\piep{\pi, \epsilon}
\def\sigep{\sigma, \epsilon}
\def\zon{Z}

\DeclareMathOperator*{\relint}{relint}
\DeclareMathOperator*{\aff}{aff}
\DeclareMathOperator*{\conv}{conv}

\DeclareMathOperator*{\IP}{IP}
\DeclareMathOperator*{\bIP}{\overline{IP}}

\newtheorem{thm}{Theorem}[section]
\newtheorem*{thm*}{Theorem}
\newtheorem{cor}[thm]{Corollary}
\newtheorem{lem}[thm]{Lemma}
\newtheorem{prop}[thm]{Proposition}
\newtheorem{conj}{Conjecture}

\newtheorem{problem}{Problem}

\theoremstyle{definition}

\newtheorem{rem}[thm]{Remark}

\title{$h^\ast$-polynomials of zonotopes}

\author{Matthias Beck}
\address{Department of Mathematics, %
San Francisco State University, %
USA}
\email{mattbeck@sfsu.edu}
\author{Katharina Jochemko}
\address{Department of Mathematics, %
Royal Institute of Technology (KTH), %
Sweden}
\email{jochemko@kth.se}
\author{Emily McCullough}
\address{Department of Mathematics, %
San Francisco State University, %
USA}
\email{emac@mail.sfsu.edu}



\date{\today}

\begin{document}

\maketitle

\begin{abstract}
The Ehrhart polynomial of a lattice polytope $P$ encodes information about the number of 
integer lattice points in positive integral dilates of $P$.
The $h^\ast$-polynomial of $P$ is 
the numerator polynomial of the generating function of its Ehrhart polynomial. 
A zonotope is any projection of a higher dimensional cube. We give a 
combinatorial description of the $h^\ast$-polynomial of a lattice 
zonotope in terms of refined descent statistics of permutations and prove that 
the $h^\ast$-polynomial
of every lattice zonotope has only real roots and therefore unimodal 
coefficients. Furthermore, we present a closed formula for the 
$h^\ast$-polynomial of a zonotope in matroidal terms which is 
analogous to a result by Stanley (1991) on the Ehrhart polynomial. Our results 
hold not only for 
$h^\ast$-polynomials but carry over 
to general combinatorial positive valuations. Moreover, we give a complete 
description of 
the convex hull of all $h^\ast$-polynomials of zonotopes in a given dimension: 
it is a simplicial cone spanned by refined Eulerian polynomials.
\end{abstract}


\section{Introduction}\label{sec:intro}

The {\bf Ehrhart function} $\ehr _P (n)$ of a polytope $P$ records the number of integer lattice points in the $n$-th positive integer dilate of the polytope. 
If $P$ is a {\bf lattice polytope} (i.e., the vertices of $P$ have all integer coordinates), Ehrhart~\cite{ehrhartpolynomial} showed that this function is in fact a 
polynomial---the \textbf{Ehrhart polynomial} of the polytope. 

A fundamental class of polytopes are zonotopes, which make an appearance in various areas of mathematics. Besides geometry and combinatorics, they play a role, for example, in
approximation theory, optimization, and crystallography. 
Given a set of vectors $V=\{v_1,\ldots, v_n\}\in \mathbb{R}^d$,
\[
\zon \ = \ \left\{ \sum _{i=1}^n \lambda _i v_i : \, 0\leq \lambda _i \leq 1 \right\}
\]
defines the {\bf zonotope} generated by $V$, and up to translation, every zonotope is generated that way.
Stanley~\cite{stanley} showed that the Ehrhart polynomial of a 
\emph{lattice} zonotope (i.e., when $v_1, \dots, v_n \in \Z^d$) is given by the 
following beautiful combinatorial formula.

\begin{thm}[{\cite{stanley}}]\label{thm:stanleyEhrhart}
Let $\zon$ be a lattice zonotope generated by a set of vectors $V\subseteq \mathbb{Z}^d$. Then
\[
\ehr_\zon (n) \ = \ \sum_{I} g(I) \, n^{|I|}
\]
where $I$ ranges over all linearly independent subsets of $V$, and $g(I)$ denotes the greatest common divisor of all maximal minors of the matrix with column vectors~$I$.
\end{thm}

A central problem, which is wide open already in dimension 3, is to characterize Ehrhart polynomials. An important tool here is the $h^\ast$-polynomial of a $d$-dimensional
lattice polytope, which encodes its Ehrhart polynomial $\ehr_P(n)$ in the basis $\binom n d$, $\binom{n+1} d$, \dots, $\binom{n+d} d$
(we give the details in Section~\ref{sec:valuations} below). A fundamental result of Stanley \cite{stanleydecomp} says that the coefficients of the
$h^\ast$-polynomial are always nonnegative integers. This set the stage for intensive studies on the inequality relations among the coefficients of $h^*$-polynomials, which remains an active area of research.

There is an entire hierarchy of conjectures concerning unimodal
$h^\ast$-polynomials. 
(A polynomial $h(t)=\sum _{i=0} ^{d} h_i t^i$ is \textbf{unimodal} if
$
	h_0 \leq h_1 \leq \dots \leq h_k \geq \dots \geq h_d
$
for some $k \in \{ 0, 1, \dots , d \}$.)
A well-known conjecture due to Stanley \cite{stanley1989log} was
originally formulated in the language of commutative algebra and implies that the $h^\ast$-polynomial of any polytope having the IDP property has unimodal coefficients. 
(A lattice polytope $P\in \mathbb{R}^d$ has the {\bf Integer Decomposition Property (IDP)} if for all integers $n\geq 1$ and every $p\in nP\cap \mathbb{Z}^d$ there are $p_1,\ldots , p_n \in P\cap
\mathbb{Z}^d$ such that $p=p_1+\cdots +p_n$.) 
Stanley's full conjecture was that standard graded Cohen--Macaulay integral domains have unimodal $h$-vectors; it has since been disproved, though the implication concerning $h^\ast$-polynomials of polytopes having the IDP property remains open; see, e.g., \cite{braunsurvey} for
background and references.

As an important non-trivial instance of Stanley's conjecture in the above form, Schepers and Van Langenhoven \cite{svl} proved that the coefficients of the $h^\ast$-polynomial for a lattice parallelepiped are unimodal.
We follow their route and investigate, more generally, $h^\ast$-polynomials of lattice zonotopes. We give a combinatorial interpretation by showing that the $h^\ast$-polynomial of
any zonotope is a weighted sum of certain polynomials $A_1(d+1,t),\ldots, A_{d+1}(d+1,t)$ originally introduced by Brenti and Welker \cite{brentiwelker}. These polynomials record the
distribution of refined descent statistics on permutations and play a central role for computing $h$-polynomials of barycentric subdivisions; see Section~\ref{sec:typeadescent}
for a detailed definition. We give a geometric interpretation for $A_{j+1}(d+1,t)$ as the $h^\ast$-polynomial of a half-open $d$-dimensional unit cube with $j$ facets removed (Theorem \ref{thm:hvectorunitcube} below). As a corollary we obtain a new characterization of reflexive polytopes in terms of Ehrhart polynomials (Proposition \ref{prop:reflexive}). We consider, more generally, half-open parallelepipeds and,
using a result of Savage and Visontai~\cite{SavageVisontai}, we prove that the $h^\ast$-polynomial of every half-open parallelepiped is real-rooted and thus unimodal
(Corollary~\ref{cor:unimodhalfopen}). This way we also obtain a new proof of a result by Br\"and\'en~\cite{branden2006linear} (Corollary~\ref{cor:branden}).
Moreover, we show that the peak of unimodality of the $h^\ast$-vector is in the 
middle. Using half-open decompositions, we then show that our results extend to zonotopes.

Our results hold not only for counting lattice points in polytopes, but for $h^\ast$-polynomials $h^\phi (P)(t)$ with respect to arbitrary combinatorial positive valuations,
initiated and studied by Jochemko and Sanyal in~\cite{jochemko2015combinatorial}; we carefully introduce the relevant terminology in Section~\ref{sec:valuations}. Our first main theorem is the following.
    \begin{thm}\label{thm:unimodalzonotope}
        Let $\varphi$ be a combinatorially positive valuation 
        and 
        let $\zon$ be an $d$-dimensional lattice zonotope.
        Then the $h^\ast$-polynomial $h^{\phi} 
        (\zon)(t)=h_0+h_1t+\cdots +h_dt^d$ has only real roots. Moreover,
        \[
        h_{0}\leq \dots \leq h _{\frac{d}{2}} 
        \geq \dots \geq h _{d} \quad \text{ if } d \text{ is even}
        \]
        and
        \[
        h _{0}\leq \dots \leq h _{\frac{d-1}{2}} 
        \quad \text{ and } \quad h _{\frac{d+1}{2}} \geq \dots \geq 
        h _{d} \quad \text{ if } d \text{ is odd. }
        \]
    \end{thm}
Our second main result gives a simple description of the convex hull of all $h^\ast$-polynomials of $d$-dimensional lattice zonotopes. 
\begin{thm}\label{thm:hconvexhull}
    Let $d\geq 1$. The convex hull of the $h^\ast$-polynomials of all 
    $d$-dimensional lattice zonotopes (viewed as points in $\R^{ d+1 }$) and 
    the convex hull of the
$h^\ast$-polynomials of all $d$-dimensional lattice parallelepipeds are both equal to the $d$-dimensional simplicial cone
    \[
    A_1(d+1,t) + \mathbb{R}_{\geq 0} \, A_2 (d+1,t) + \cdots + \mathbb{R}_{\geq 0} \, A_{d+1} (d+1,t) \, .
    \]
\end{thm}

Our third line of research concerns type-$B$ zonotopes and coloop-free zonotopes, introduced in Section~\ref{sec:altincr}, which we believe are both interesting in their  own right.
Schepers and Van Langenhoven~\cite{svl} conjectured that every polytope having the IDP property and an interior lattice point has an alternatingly increasing
$h^\ast$-polynomial, a property stronger than unimodality. We give further evidence for their conjecture by proving it for type-$B$ and coloop-free zonotopes
(Corollary~\ref{cor:Ehrhart_lcs} and Theorem~\ref{thm:coloopfree}). We relate the
$h^\ast$-polynomial of type-$B$ zonotopes to type-$B$ Eulerian polynomials via discrete geometry. Again our results hold in the more general context of translation-invariant
valuations.  We introduce refined type-$B$ Eulerian polynomials, and by expressing them in
terms of $A_1(d+1,t),\ldots, A_{d+1}(d+1,t)$ we prove that these refined 
Eulerian polynomials are real-rooted 
(Theorem~\ref{thm:typeBalternatinglyincr}), generalizing a result of Brenti~\cite{brenti_qEulerian}. 

Our final main result is a closed formula for the 
$h^\ast$-polynomial of a lattice zonotope in the spirit of Theorem 
\ref{thm:stanleyEhrhart}.
\begin{thm}\label{thm:zonotopevaluation}
    Let $Z$ be a $d$-dimensional lattice zonotope generated by a set of vectors 
    $V\subset \mathbb{Z}^d$, and let $\varphi$ be a 
    translation-invariant valuation. Then 
    \[
    h^\phi(Z)(t)
    \ = \ \sum _{I\in \mathcal{I}}b_\phi(I)\sum_{B\in \mathcal{B}\atop 
    B\supseteq  
        I}A_{|I\cup \IP 
        (B)|+1}(d+1,t) \, .
    \]
\end{thm}
Here, $\mathcal{I}$ and $\mathcal{B}$ denote the set of independent subsets of 
$V$ and the bases formed by elements in $V$, respectively.
The internally passive elements of a basis $B$ (see Section 
\ref{sec:matroidalasp} for a definition) are collected in 
$\IP (B)$, and $b_\phi (I)$ is the value of 
$\phi$ on the relative interior of the parallelepiped generated by the vectors 
in~$I$.


\section{Preliminaries}
\subsection{Polynomials}
A polynomial $h(t)=\sum _{i=0} ^{d} h_i t^i$ of degree $d$ is called 
\textbf{unimodal} if its 
coefficient vector ${\bf h }=(h_0,\ldots,h_d)$ is unimodal,
that is, if
\[
	h_0 \leq h_1 \leq \dots \leq h_k \geq \dots \geq h_d
\]
for some $k \in \{ 0, 1, \dots , d \}$;
we say that $h(t)$ and ${\bf h}$ have a 
\textbf{peak} at $k$. The polynomial $h(t)$ is called \textbf{alternatingly 
increasing} if
\[
h_0 \le h_d \le h_1 \le h_{d-1} \le \dots \le h_{\floor{\frac{d+1}{2}}} \, .
\]
In particular, if $(h_0, h_1, \dots, 
h_d)$ is alternatingly increasing, then $(h_0, h_1, \dots, h_d)$ is unimodal 
with peak at $\floor{\frac{d+1}{2}}$. We call a polynomial $h(t)$ 
\textbf{palindromic} with 
center of symmetry at $\frac{d}{2}$ if $t^dh(\frac{1}{t})=h(t)$. If it is in 
addition unimodal, then the  coefficients closest to the center of symmetry are 
maximal, i.e., $h(t)$ has a peak at $\frac{d}{2}$ if $d$ is even, and at 
$\lfloor \frac{d}{2}\rfloor$ and $\lfloor \frac{d}{2}\rfloor +1$ if $d$ is odd. 

Every polynomial $h(t)$ of degree $d$ can be uniquely decomposed into a sum 
$h(t)=a(t)+t \, b(t)$, where $a(t)$ and $b(t)$ are palindromic with 
$t^da(\frac{1}{t})=a(t)$ and $t^{d-1}b(\frac{1}{t})=b(t)$~\cite{stapledondelta}. 
\begin{lem}\label{lem:decomp_poly}
        Let $h(t)=a(t)+t \, b(t)$ be a polynomial of degree $d$, where 
        $a(t)$ and $b(t)$ are palindromic with center of symmetry $\frac{d}{2}$ 
        and $\frac{d-1}{2}$, respectively. Then $h(t)$ is alternatingly 
        increasing if and only if $a(t)$ and $b(t)$ are unimodal.
\end{lem}
\begin{proof}
Let $h(t)=\sum_{i=0}^{d}h_it^i$. Since $a(t)$ and $b(t)$ are palindromic, it 
is easy to check that $h_i\leq 
h_{d-i}$ for all $i$ if and only if $b(t)$ is unimodal, and  $h_{d-i}\leq 
h_{i+1}$ for all $i$ if and only if $a(t)$ is unimodal.
\end{proof}


\subsection{Translation-invariant valuations}\label{sec:valuations}
A \textbf{lattice polytope} is a polytope in $\R^d$ with vertices in the integer lattice $\Z^d$. 
The family of all lattice polytopes in $\mathbb{R}^d$ will be denoted by 
$\Pol{\Z^d}$. A \textbf{valuation} on lattice polytopes is a map $\varphi$ from 
$\Pol{\Z^d}$ into some Abelian group $G$ such that $\phi (\emptyset)=0$ and
\[
\varphi (P\cup Q) \ = \ \varphi (P) + \varphi (Q) - \varphi (P\cap Q)
\]
whenever $P,Q,P\cup Q, P\cap Q \in \Pol{\Z^d}$. In 
\cite{mcmullen2009valuations} McMullen showed that every 
valuation satisfies the \textbf{inclusion-exclusion property}. Namely, for 
lattice 
polytopes $P_1,\ldots,P_m \in \Pol{\Z^d}$ such that $P_1\cup \cdots \cup P_n 
\in \Pol{\Z^d}$ and $\bigcap _{i\in I}P_i \in \Pol{\Z^d}$ for all $I\subseteq 
[m] := \{ 1, 2, \dots, m \}$,
\[
\varphi (P_1\cup \cdots \cup P_m) \ = \ \sum _{\emptyset \neq I} (-1)^ {|I|-1} \, \varphi \left(\bigcap _{i\in I}P_i \right).
\]
This allows for a definition of $\varphi$ on the relative interior $\relint P$ of a polytope as
\[
\varphi(\relint P) \ = \ \sum _{F} (-1)^{\dim P -\dim F}\varphi (F) \, ,
\]
where the sum is taken over all faces of $P$. We call $\varphi$  
\textbf{translation-invariant} or a \textbf{$\Z^d$-valuation} if 
$\varphi(P+x)=\varphi(P)$ for all $x\in \Z^d$ and all $P\in \Pol{\Z^d}$. 
Fundamental examples besides the volume are the Euler characteristic, the 
\textbf{discrete volume} 
$\ehrval(P):=|P\cap \Z^d|$ and 
the solid-angle sum (see, e.g., \cite{solidangle}). McMullen 
\cite{mcmullen1977valuations} proved that for integers $n\geq 0$ 
the value $\varphi (nP)$ of the $n$-th dilate of an $r$-dimensional lattice 
polytope $P$ is given by a polynomial $\ehr^\phi_P (n)$ of degree at most $r$ in 
$n$. For the discrete volume this was proved by Ehrhart 
\cite{ehrhartpolynomial}; when $\phi = \ehrval$, we suppress the superscript 
and call $\ehr_P (n)$ the {\bf Ehrhart
polynomial} of $P$. 
Equivalently, there are $h_0^\phi (P),\ldots, h_r^\phi (P) \in G$ such that  
\[
\ehr^\phi_P (n)\ = \ h_0^\phi (P) {n+r \choose r} + h_1^\phi (P) {n+r-1 \choose r} + \cdots + h_r^\phi (P){n \choose r}
\]
for all $n\geq 0$.
In terms of generating series, this is equivalent to
\[
\Ehr^\phi (P,t) \ := \ \sum _{n\geq 0} \ehr^\phi_P(n) \, t^n \ = \ \frac{h_0^\phi (P)+ \cdots + h_r^\phi (P) \, t^r}{(1-t)^{r+1}} \, .
\]
In the special case $\phi = \ehrval$, we call $\Ehr (P, t)$ the {\bf Ehrhart series} of~$P$.
The numerator polynomial $h^\phi (P)(t)$ is called the 
\textbf{$h^\ast$-polynomial of $P$ with 
respect to $\varphi$} and the vector $h^\phi (P) := (h_0^\phi (P), 
\ldots,h_r^\phi (P))$ is the \textbf{$h^\ast$-vector of $P$ with respect to $\phi$}.
When $\phi = \ehrval$, we call $h^\phi (P)$ simply the \textbf{$h^\ast$-vector}; alternative names in this case include 
\textbf{$\delta$-vector} and \textbf{Ehrhart $h$-vector}. 

For the discrete volume 
Stanley~\cite{stanleydecomp} showed that the entries of  $h^\phi (P)$ are 
nonnegative 
for all lattice polytopes $P$. For the solid-angle sum this was shown by Beck, 
Robins and Sam \cite{solidangle}. In \cite{jochemko2015combinatorial} 
Jochemko and Sanyal studied the class of all 
$\mathbb{Z}^d$-valuations into some partially 
ordered Abelian group such that $h^\phi (P)(t)$ has nonnegative entries 
for every lattice polytope $P$. They called these valuations 
\textbf{combinatorially positive} and 
obtained the following simple characterization.
\begin{thm}[{\cite{jochemko2015combinatorial}}]\label{thm:combpositive}
Let $\phi$ be a $\Z^d$-valuation. Then
$\phi$ is combinatorially positive if and only if
$\phi (\relint \Delta)\geq 0$ for all simplices $\Delta \in \Pol{\Z^d}$.
\end{thm}

Note that this implies that $\phi (\relint P)\geq 0$ for all $P \in \Pol{\Z^d}$ if $\phi$ is combinatorially positive. 


\subsection{Half-open polytopes}
To every polytope $P\in \mathcal{P}(\mathbb{Z}^d)$ and every generic 
$q$ in the affine hull $\aff(P)$ of $P$ we can associate a half-open polytope 
$\half{q}{P}$, defined as the set of points $p\in P$ such that $[q,p)\cap 
P\neq \emptyset$. Thinking of $q$ as a light source, $\half{q}{P}$ is the set 
of all points in $P$ that are not visible from $q$. Note that $\half{q}{P}$ is 
closed if and only if $q\in P$ and in this case $\half{q}{P}=P$. If 
$F_1,\ldots, F_m$ are the facets of $P$, let $I_q(P)\subseteq [m]$ contain the indices 
of facets visible from $q$. Then
\[
\half{q}{P} \ = \ P \setminus \bigcup _{i\in I_q(P)} F_i \, .
\]
Accordingly, for a valuation $\varphi$ we define
\[
\phi \left(\half{q}{P}\right) \ := \ \phi (P) - \sum _{\emptyset\neq I\subseteq 
[m]} (-1)^{|I|-1}\phi 
\left(\bigcap _{i\in I}F_i\right).
\]
In particular, we can consider $\ehr^\phi_{\half{q}{P}}(n)$ and the accompanying
$h^\ast$-polynomial of $\half{q}{P}$. For example, it is easy to see that if 
$Q$ is a half-open unimodular simplex\footnote{
A {\bf simplex} is a $d$-polytope with (the minimal number of) $d+1$ vertices; it is {\bf unimodular} if these vertices have integer coordinates and the simplex has (minimal) volume $\frac{ 1 }{ d! }$. 
}
of dimension $d$ with 
$k$ missing (visible) facets, then $\ehr_{Q}(n)={n+d-k \choose d}$, or 
equivalently, $\Ehr(Q, t) = \frac{ t^k }{ (1-t)^{ d+1 } }$.

\begin{lem}[{\cite[Theorem~3]{KV}}]\label{lem:ho}
    Let $P$ be a polytope, $P = P_1 \cup \cdots \cup P_k$ a dissection
    and $q \in \aff(P)$ generic. Then
    \[
    \half{q}{P} \ = \ 
    \half{q}{P_1} \uplus 
    \cdots \uplus 
    \half{q}{P_k}
    \]
    is a disjoint union of half-open polytopes.
\end{lem}

\begin{cor}[{\cite[Corollary 3.2]{jochemko2015combinatorial}}]\label{cor:ho-val}
    Let $P = P_1 \cup \cdots \cup P_k$ be a dissection with $P_1,\dots,P_k \in
    \mathcal{P}(\Z^d)$. If $\phi$ is a valuation, then for a generic $q \in
    \relint(P)$
    \[
    \phi(P) \ = \ 
    \phi(\half{q}{P_1}) + 
    \cdots + 
    \phi(\half{q}{P_k}) \, .
    \]
\end{cor}


\section{Descent Statistics}
\subsection{Type-$A$}\label{sec:typeadescent}
Let $S_d$ denote the set of all permutations on $[d]$. For a
permutation word $\sigma=\sigma_1\sigma_2\cdots\sigma_d$ in $S_d$ the 
\textbf{descent set} is defined by 
\[
	\Des (\sigma) \ := \ \{i\in [d-1]  : \,  \sigma_i > \sigma_{i+1}\} \, .
\]
The number of descents of $\sigma$ is denoted by $\des(\sigma):=\left| \Des (\sigma)\right|$. The \textbf{(type-$A$) Eulerian number} $a(d,k)$ counts the number of
permutations in $S_d$ with exactly $k$ descents: 
\[
	a(d,k) \ := \ \left| \lbrace \sigma \in S_d  : \,  \des (\sigma) = k\rbrace\right| .
\] 
We consider a refinement of the descent statistic: the \textbf{$(A,j)$-Eulerian number} 
\[
	a_j(d,k) \ := \ \left| \left\{ \sigma \in S_d  : \,  \sigma_d = d+1-j \text{ and } \des(\sigma) = k \right\} \right|
\]
giving the number of permutations $\sigma \in S_d$ with last letter $d+1-j$ and exactly $k$ descents. The corresponding \textbf{$(A,j)$-Eulerian polynomial} is
\[
	A_j (d,t) \ := \ \sum _{k=0} ^{d-1} a_j (d,k) \, t^k \, .
\]
Note that by definition $a_j (d,k)=0$ for $k<0$, $j<1$, $k>d-1$ and $j>d$. As far as we know, the 
$(A,j)$-Eulerian polynomials were first considered by Brenti and 
Welker~\cite{brentiwelker}, though the $(A,j)$-Eulerian \emph{numbers} and generalizations of 
them were considered earlier (see, e.g., 
\cite{ehrenborg1998mixed,stanley1981two}).


\subsection{Type-$B$}
A {\bf signed permutation} on $[d]$ is a pair $(\sigep)$ with $\sigma \in S_d$ and $\epsilon \in \{ \pm1 \}^d$.
To each letter $\sigma_{i}$ in the permutation word $\sigma$ we assign the sign 
$\epsilon_{i}$, the $i$-th entry of $\epsilon$.
For a given $d$, the set of signed permutations  is denoted by $B_d$ and has $2^d \, d!$ elements.
We will use one-line notation to denote signed permutation words with the convention that letters associated with a negative sign will be followed by an accent mark. So for $d=5$, $\sigma = 42135$ and $\epsilon = (-1, -1, 1, -1, 1)$ we write $(\sigep) = 4' 2' 1 3' 5$.

Set $\sigma_0:=0$ and $\epsilon_0:=1$ for all $(\sigep) \in B_d$ and all $d \ge 1$. Then $i \in [d-1] \cup \{0\}$ is a {\bf descent} of $(\sigep) \in B_d$ if $\epsilon_{i} \sigma_i > \epsilon_{i+1} \sigma_{i+1}$.
E.g., 0 and 3 are the descents of $4' 2' 1 3' 5$. We define the {\bf descent set} and the {\bf descent number} of $(\sigep) \in B_d$, respectively, as
\begin{align*}
	\Des(\sigep) \ &:= \ \left\{ i \in [d-1] \cup \{ 0 \} : \, \epsilon_{i} \sigma_i > \epsilon_{i+1} \sigma_{i+1} \right\} \ \text{ and } \\
	\des(\sigep) \ &:= \ \left| \Des(\sigep) \right| .
\end{align*}
For a general background on type-B descents, see, e.g., \cite{brenti_qEulerian}.
We observe that the descent statistic on permutations in $S_d$ agrees with the 
descent statistic on signed permutations $B_d$ when we fix the sign vector 
$\epsilon = {\bf 1} := (1, 1, \dots, 1)$. However, since $0$ is a possible 
descent of a signed permutation, $0 \leq \des(\sigep) \leq d$ for all 
$(\sigep) \in B_d$, in contrast to $0 \leq \des(\sigma, {\bf1}) \leq d-1$ for 
all $\sigma \in S_d$.

The number of signed permutations on $[d]$ with exactly $k$ descents is the 
{\bf type-$B$ Eulerian number}. 
We write
\[
	b(d,k) \ := \ \left| \left\{ (\sigep) \in B_d : \, \des(\sigep) = k \right\} \right| .
\]
The {\bf type-$B$ Eulerian polynomial} is
\[
	B(d,t) \ := \ \sum_{k=0}^d b(d,k) \, t^k .
\]
Analgous to the type-$A$ case, we introduce the {\bf $(B,l)$-Eulerian numbers}, a refinement of the type-$B$ Eulerian numbers, defined by
\[
b_{l}(d,k) \ := \ \left| \left\{ (\sigep) \in B_{d}: \, \epsilon_{d} \sigma_{d} = d+1-l \text{ and } \des(\sigep)=k \right\} \right| ,
\]
where $1 \le l \le d$, and define the {\bf $(B,l)$-Eulerian polynomial}
\[
B_{l} (d,t) \ := \ \sum_{k=0}^d b_{l} (d,k) \, t^k .
\]
As far as we know, these have not been studied before.

\subsection{Unimodality and real-rootedness}
A fundamental result of Savage and Visontai~\cite{SavageVisontai} implies that 
the $(A,j)$-Eulerian 
polynomials have only real roots and are therefore unimodal. In fact, 
they proved the following stronger result.
\begin{thm}[{\cite{SavageVisontai}}]\label{thm:jEulerianrealrooted}
Let $c_1,\ldots,c_d \geq 0$ be real numbers. Then the polynomial
\[
c_1A_1(d,t) + c_2A_2(d,t) + \cdots + c_d A_d (d,t)
\]
has only real roots. In particular, its coefficients form a unimodal sequence.
\end{thm}
Their inductive proof was based on the following recurrence for 
$(A,j)$-Eulerian polynomials, which seems to go back to Brenti and 
Welker~\cite{brentiwelker}. 
\begin{lem}[{\cite[Lemma 2.5]{brentiwelker}}]
    \label{lem:jEulerian_recursion}
    For $1\leq j \leq d+1$,
    \begin{align*}
    A_j (d+1, t) \ &= \ t \sum _{l=1} ^{j-1} A_l (d, t) + \sum _{l=j} ^d A_l (d,t) 
    \, .
    \end{align*}
\end{lem}
Note that in general, $A_j(d,t)$ is not palindromic. Nevertheless, using the 
recurrence above together with the following lemma one can
determine the exact position of their peaks.
\begin{lem}[{\cite[Lemma 2.5]{brentiwelker}}]
    \label{lem:jEulerian_symmetry}
    For all $d \ge 1$ and $1 \le j \le d$, 
    \begin{align*}		
    A_j(d,t) \ &= \ t^{d-1} A_{d+1-j} \left(d, \tfrac{1}{t} \right) .
    \end{align*}
\end{lem}
The following theorem is a slight strengthening of \cite[Corollary 
4.4]{kubitzke2009lefschetz} by Kubitzke and Nevo. While they used quite heavy 
algebraic machinery, we give an elementary combinatorial proof.
\begin{thm}
\label{thm:jEulerian_unimodality}
For all $1\leq j\leq d$, the coefficients of $A_j(d,t)$ are unimodal. More specifically, if $d$ is even,
\[
\begin{array}{cccccccccc}
a_j(d,0)	& \leq 	& \cdots 	& \leq 	& a_j(d,\frac{d}{2}-1)	& \geq	& \cdots 	&\geq 	& a_j (d,d-1)	& \text{ if } \ 1\leq j \leq \frac{d}{2} \, ,\\
a_j(d,0)	& \leq 	& \cdots 	& \leq 	& a_j(d,\frac{d}{2})		& \geq 	& \cdots 	&\geq 	& a_j (d,d-1)	& \text{ if } \ \frac{d}{2}< j \leq d \, ,
\end{array}
\]
and if $d \ge 3$ is odd,
\[
\begin{array}{ccccccccccc}
a_1(d,0)	& \leq 	& \cdots 	& \leq 	& a_1(d,\lfloor \frac{d}{2}\rfloor-1)	
& =		& a_1(d,\lfloor \frac{d}{2}\rfloor)	& \geq 	& \cdots 	&\geq 	& a_1 
(d,d-1)\\
a_d(d,0)	&\leq 	& \cdots 	& \leq 	& a_d(d,\lfloor \frac{d}{2}\rfloor)	& =		& a_d(d,\lfloor \frac{d}{2}\rfloor +1)	& \geq 	& \cdots 	&\geq
& a_d (d,d-1) \, ,
\end{array}
\]
\[
\begin{array}{cccccccccc}
a_j(d,0)	& \leq 	& \cdots 	& \leq 	& a_j(d,\lfloor \frac{d}{2}\rfloor)		& \geq 	& \cdots 	& \geq 	& a_j (d,d-1)	& \text{ if } \ 2\leq j\leq d-1 \, .
\end{array}
\]
\end{thm}

\begin{proof}
We argue by induction on $d$.
When $d=1$, the claim is trivially true since $A_1(1,t)=1$. The case $d=2$ is 
easily checked. 

Let $d+1$ be even. We distinguish two cases:

\textsc{Case:} $1\leq j \leq \frac{d+1}{2}$.
Then
\[
	A_j (d+1, t) \ = \ t \sum _{l=1} ^{j-1} A_l (d, t) + \sum _{l=j} ^{d+1-j} A_l (d, t) + \sum _{l=d+2-j} ^d A_l (d, t)
\]
by Lemma~\ref{lem:jEulerian_recursion}. The first and the third summand added give, by Lemma~\ref{lem:jEulerian_symmetry}, a palindromic polynomial with center of symmetry at $\frac{d}{2}$ which, by induction, has unimodal coefficients with peaks at 
$\floor{\frac{d}{2}}$ and $\floor{\frac{d}{2}}+1$.
The second summand has, by induction, unimodal coefficients with peak at 
$\floor{\frac{d}{2}}=\frac{d+1}{2}-1$.

\textsc{Case:} $\frac{d+1}{2} <  j \leq d+1$.
Then
\[
	A_j (d+1, t) \ = \ t \sum _{l=1} ^{d+1-j} A_l (d, t) + t\sum _{l=d+2-j} ^{j-1}A_l (d, t) + \sum _{l=j} ^d A_l (d, t) \, .
\]
The first and the third summand added give a palindromic polynomial with center 
of symmetry at $\frac{d}{2}$, which has unimodal coefficients with peaks at 
$\floor{\frac{d}{2}}$ and $\floor{\frac{d}{2}}+1$. The 
coefficients of the second summand form a unimodal sequence with peak at 
$\floor{\frac{d}{2}}+1=\frac{d+1}{2}$.

If $d+1 \ge 3$ is odd, we distinguish again two cases.

\textsc{Case:} $1 \le j \le \frac{d+1}{2}$.
By Lemma~\ref{lem:jEulerian_recursion},
\[
	A_j (d+1, t) \ = \ t \sum _{l=1}^{j-1} A_l(d, t) + \sum_{l=j}^{d+1-j} A_l (d, t) + \sum_{l=d+2-j}^d A_l(d, t) \, .
\]
The second summand is, by induction and Lemma~\ref{lem:jEulerian_symmetry}, a palindromic polynomial with unimodal coefficients and peaks at $\frac{d}{2}-1$ and $\frac{d}{2}$. The coefficients of the first and third summand are unimodal with peak at 
$\frac{d}{2} = \floor{\frac{d+1}{2}}$.

\textsc{Case:} $\frac{d+1}{2} < j \le d+1$.
Then
\[
	A_j (d+1, t) \ = \ t \sum _{l=1}^{d+1-j} A_l (d, t) + t\sum_{l=d+2-j}^{j-1} A_l(d, t) + \sum _{l=j}^d A_l(d, t) \, .
\]
As in the previous case, the coefficients of the summand in the middle are unimodal and palindromic, this time with peaks at $\frac{d}{2}$ and $\frac{d}{2}+1$. The coefficients of the first and third summand form again a unimodal sequence with peak at $\frac{d}{2}=\floor{\frac{d+1}{2}}$.
\end{proof}
From the proof of \cite[Proposition~2.17]{svl}, it can moreover be seen that 
the coefficients of $A_j(d,t)$ are alternatingly increasing for 
sufficiently large $j$. We formally record this result and give a short proof.
\begin{lem}
\label{lem:jEulerian_alt_incr}
For all $d \ge 0$ and $\frac{d+1}{2} < j \le d+1$, the coefficients of 
$A_{j}(d+1,t)$ are alternatingly increasing.
\end{lem}
\begin{proof}
If $\frac{d+1}{2} < j \le d+1$ then, by Lemma 
\ref{lem:jEulerian_recursion}, we have $A_j (d+1, t) = b(t) + t \, c(t)$
with
\[
	b(t) \ = \ t \sum _{l=1} ^{d+1-j} A_l (d, t) + \sum _{l=j} ^d A_l (d, t)
\]	
and
\[
	c(t) \ = \sum _{l=d+2-j} ^{j-1}A_l (d, t) \, .
\]
By Lemmas~\ref{lem:jEulerian_symmetry} and~\ref{thm:jEulerian_unimodality} 
we know 
$b(t)$ and $c(t)$ are unimodal and $t^d \, b(\frac{1}{t}) = b(t)$ and 
$t^{d-1} c(\frac{1}{t}) = c(t)$. Therefore the claim follows with 
Lemma~\ref{lem:decomp_poly}. 
\end{proof}

\begin{rem}
    Ehrenborg, Readdy, and Steingr{\'i}mmson showed in \cite{ehrenborg1998mixed} 
    that the $(A,j)$-Eulerian numbers have a geometric meaning as mixed volumes 
    of certain hypersimplices. It would be interesting to see whether this 
    yields a 
    geometric proof of Theorem~\ref{thm:jEulerian_unimodality} by using, e.g.,
    the Alexandrov--Fenchel inequalities.
\end{rem}
Brenti~\cite{brenti_qEulerian} proved that the type-$B$ Eulerian polynomials have
only real roots. 
\begin{thm}[{\cite[Corollary 3.7]{brenti_qEulerian}}]\label{thm:typeBrealrooted}
    The type-$B$ Eulerian polynomial $B(d,t)$ has only real roots. In particular, the 
    coefficients of $B(d,t)$ form a unimodal sequence.
\end{thm}
In Section~\ref{sec:typeBzonotopes} we prove the following explicit 
expression of $(B,l)$-Eulerian polynomials in terms of $(A,j)$-Eulerian 
polynomials.
\begin{prop}\label{thm:typeB} For all $0\leq l\leq d$
 \begin{align*}
 B_{l+1}(d+1,t) \ &= \ 2^l\sum_{j=0}^{d-l} \binom{d-l}{j} A_{j+l+1}(d+1, t) \, .
 \end{align*}
\end{prop}
We obtain the following generalization of Theorem~\ref{thm:typeBrealrooted} as 
a corollary, which to our knowledge is new.
\begin{thm}\label{thm:typeBalternatinglyincr}
The polynomial $B_l(d,t)$ has only real roots and is alternatingly increasing for all $1\leq l\leq d$.
\end{thm}
\begin{rem}
Theorem \ref{thm:typeBalternatinglyincr} can be generalized further. Indeed, 
we could have defined $(B,l)$-Eulerian numbers for integers $l\in [2d+1]\setminus \{d+1\}$. For $l\in [d]$ the result can also be seen from a bijection of Pensyl and Savage \cite[Theorem 3]{pensyl2013lecture} 
between signed permutations and $\bf{s}$-lecture hall partitions for ${\bf s}=(2,4,\ldots,2d)$ together with the results from \cite{SavageVisontai} and \cite[Remark 6.11]{beckbrauneulerian}. By changing the sign of every letter of a signed permutation we obtain $B_l(d,t)=t^dB_{2d+2-l}(d,\frac{1}{t})$ and thus real-rootedness for $d+1<l\leq 2d+1$.
\end{rem}
\begin{lem}\label{lem:altincr}
For all $d\geq 1$ and $1\leq i,j\leq d$ with $i+j\geq d+2$,
\[ A_i(d+1,t)+A_j(d+1,t) \]
is alternatingly increasing.
\end{lem}
\begin{proof}
If both $i$ and $j$ are greater than $ \frac{d+1}{2}$ then the claim 
follows from Lemma~\ref{lem:jEulerian_alt_incr}. So we may suppose that $i\leq 
\frac{d+1}{2}$; then we must have $j=d+2-i+k>\frac{d+1}{2}$ for some $0\leq 
k\leq i-1$ by assumption. Using the recursions in the proof of Theorem 
\ref{thm:jEulerian_unimodality} we obtain
\begin{align*}
&A_i(d+1,t)+A_j(d+1,t) \ = \ t\sum_{l=1}^{i-1}A_l(d,t)+\sum_{l=i}^{d+1-i}A_l(d,t)+\sum_{l=d+2-i}^{d}A_l(d,t)+ t\sum_{l=1}^{i-k-1}A_l(d,t)\\
&\qquad +t\sum_{l=i-k}^{i-1}A_l(d,t)+t\sum_{l=i}^{d-i+1}A_l(d,t)+t\sum_{l=d-i+2}^{d+k-i+1}A_l(d,t)+
\sum_{l=d+k-i+2}^{d}A_l(d,t) \, .
\end{align*}
As in the proof of Theorem 
\ref{thm:jEulerian_unimodality}, the first, third, fourth, and 
last summand add up to a palindromic polynomial with center of symmetry at 
$\frac{d}{2}$. By Lemma~\ref{lem:jEulerian_symmetry} this is also true for the 
sum of the second and the sixth sum. The remaining sums yield a polynomial 
that is palindromic with center of symmetry at $\frac{d-1}{2}$ times a factor 
$t$. 
Therefore the claim follows from Lemma~\ref{lem:decomp_poly}.
\end{proof}
\begin{proof}[Proof of Theorem~\ref{thm:typeBalternatinglyincr}]
By Proposition~\ref{thm:typeB},
\[
 B_{l+1}(d+1,t) \ = \ 2^{l-1}\sum_{j=0}^{d-l} \binom{d-l}{j} \Big( A_{j+l+1}(d+1, t)+A_{d-j+1}(d+1,t) \Big) \, .
\]
By Lemma~\ref{lem:altincr}, $A_{j+l+1}(d+1,t)+A_{d-j+1}(d+1,t)$ is alternatingly increasing for all $0\leq j\leq d-l$, and so is $B_{l+1}(d+1,t)$, since it is a positive linear combination of polynomials of that form. Real-rootedness follows from Theorem~\ref{thm:jEulerianrealrooted}.
\end{proof}

\section{Geometry}
\subsection{Half-open unit cubes}	
For $j \in \lbrace 0,\ldots,d\rbrace$ we define the half-open unit cube
\[
    C_j ^d \ := \ [0,1]^d\setminus \lbrace \x \in \mathbb{R}^d : \,  x_d = 
    x_{d-1} = \dots = x_{d+1-j}=1\rbrace \, .
\]
The subscript $j$ indicates the number of facets removed from $C^d := [0,1]^d$. The {\bf $j$-descent set} $\Des _j (\sigma)\subseteq \left\{ 1, \dots, d \right\}$ of a permutation $\sigma \in S_d$ is
\[
\Des _j (\sigma) \ := \
	\begin{cases}
		\Des (\sigma)\cup \left\{ d \right\}	& \text{ if }  d+1 - j \leq \sigma_d \leq d , \\
		\Des (\sigma)  					& \text{ otherwise} , \\
	\end{cases}
\]
and the {\bf $j$-descent number} $\des _j (\sigma) := |\Des_j(\sigma)|$ counts the $j$-descents of $\sigma$.
We can describe the $(A,j)$-Eulerian numbers in terms of $j$-descents, as the following lemma shows.
\begin{lem}
\label{lem:jdescents}
For all $0\leq j\leq d$ and $0\leq k\leq d$,
\[
	\left| \left\{ \sigma \in S_d: \,\des _j (\sigma)=k\right\} \right| \ = \ a_{j+1} (d+1,k) \, .
\]
\end{lem}
\begin{proof}
For every $\sigma  \in S_d$ define $\sigma ' \in S_{d+1}$ by
\[
\sigma_i' \ := \ \begin{cases}
\sigma_i & \text{ if }\sigma_i < d+1-j,\\
\sigma_i+1 & \text{ if }\sigma_i \geq d+1-j,\\
\sigma_{d+1}=d+1-j & \text{ otherwise.}
\end{cases}
\]
It is straightforward to check that the map $\sigma \mapsto \sigma'$ 
bijectively maps $S_d$ to the set of permutations on $[d+1]$ ending with 
$d+1-j$, and thus $\des_j(\sigma)=\des(\sigma ')$.
\end{proof}
The discrete volume of half-open unit cubes plays a distinguished role when 
determining the $h^\ast$-vector of parallelepipeds with respect to arbitrary 
$\Z^d$-valuations; this is based on the following result.

\begin{thm}\label{thm:hvectorunitcube}
Let $0 \le j \le d$. Then
\[
  \Ehr \left( C_j^d, t \right)
  \ = \ \frac{A_{j+1}(d+1, t)}{(1-t)^{d+1}} \, .
\]
In particular, the $h^\ast$-polynomial is real-rooted and its coefficients 
form a unimodal sequence.
\end{thm}
\begin{proof}
We will decompose $C^d_j$ into half-open unimodular 
simplices induced by the arrangement of hyperplanes given by inequalities of 
the form $x_i=x_j$, $i\neq j$; see also~\cite[Chapter 6]{BS}.
\begin{figure}
    \begin{tikzpicture}[scale=0.3]
    \draw[->] (0,-1) -- (0,8.5);
    \draw[->] (-1,0) -- (8.5,0);
    \draw (-0.1,6.5) -- (0,6.5) node[left=1pt] {$1$};
    \draw (-1,8) node {$x_2$};
    \draw (8,-1) node {$x_1$};
    \draw[thick] (6.5,6.3) -- (6.5,0);
    \draw[thick] (0.2,0.2) -- (6.3,6.3);
    \draw[thick] (0.2,0) -- (6.5,0);
    \draw[dotted,thick] (0,6.3) -- (6.3, 6.3);
    \draw[dotted,thick] (0.3,0) -- (6.5, 6.2);
    \filldraw[fill=red!30,draw=red!40,opacity=0.7] (0,0) --(0,6.3)-- 
    (6.3,6.3)--(0,0);
    \filldraw[fill=blue!30,draw=blue!40,opacity=0.7] (0.3,0) --(6.5, 6.2)-- 
    (6.5,0)--(0.3,0);
    \draw[black] (4.5,2) node {$(12)$};
    \draw[black] (2,4.5) node {$\id$};
    \end{tikzpicture}
\caption{Decomposition of $C_1^2$ into half-open unimodular simplices.}
\end{figure}
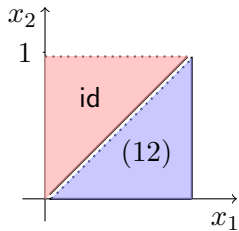
For $0 \le j \le d$ and $\sigma \in S_d$, define the half-open unimodular 
simplex
\begin{align*}
\triangle_{\sigma}^{d, j}
\ :=& \
\begin{Bmatrix}
\x \in C^d_j : \, x_{\sigma_1} \le x_{\sigma_2} \le \dots \le x_{\sigma_d} 
\\[0.3em]
\text{ with } x_{\sigma_i} < x_{\sigma_{i+1}} \text{ when } i \in \Des(\sigma) 
\\[0.3em]
\end{Bmatrix} \\
=& \
\begin{Bmatrix}
\x \in \R^d : \, 0 \le x_{\sigma_1} \le x_{\sigma_2} \le \dots \le x_{\sigma_d} 
\le 1 \\[0.3em]
\text{ with } x_{\sigma_i} < x_{\sigma_{i+1}} \text{ when } i \in 
\Des_{j}(\sigma) \\[0.3em]
\text{ and } x_{\sigma_d} < 1 \text{ when } d \in \Des_{j}(\sigma) \\[0.3em]
\end{Bmatrix} \, .
\end{align*}
The closure of $\triangle_{\sigma}^{d, j}$ is a unimodular simplex for all $j$ 
and $\sigma$. Each strict inequality corresponds bijectively to a missing facet 
of the simplex. Therefore, the half-open unimodular 
simplex $\triangle_{\sigma}^{d, j}$ has exactly $\des_j(\sigma)$ missing facets.
Furthermore,
\[
C^d_j \ = \ \biguplus_{\sigma \in S_d} \triangle_{\sigma}^{d, j} \, 
\]
is a disjoint union.
Therefore, 
\[
\Ehr \left( C_j^d, t \right)
\ = \ \sum_{\sigma \in S_d} \Ehr \left( \triangle_{\sigma}^{d, j} , t \right) \ 
= \ \frac{ \sum_{\sigma \in S_d} t^{\des_j(\sigma)} }{(1-t)^{d+1}} \ = \ 
\frac{A_{j+1}(d+1, t)}{(1-t)^{d+1}} \, , 
\]
where the last equality holds by Lemma~\ref{lem:jdescents}.
\end{proof}
The $h^\ast$-vector of 
parallelepipeds with respect to arbitrary $\Z^d$-valuations will 
be treated in the following section.


\subsection{Half-open parallelepipeds}\label{sec:halfopenparallelepiped}
In the following let $\varphi$ be a $\mathbb{Z}^d$-valuation, and let 
$v_1,\ldots,v_r \in \mathbb{Z}^d$ be fixed linearly independent vectors.
For each $I\subseteq [r]$ we define the (closed) parallelepiped
\[
\lozenge (I) \ := \ \left\{\sum _{i\in I} \lambda _i v_i : \,  0\leq \lambda _i \leq 
1 \text{ for all } i\in I \right\}
\]
and the relatively open parallelepiped 
\[
\Box (I) \ := \ \left\{\sum _{i\in I} \lambda _i v_i : \,  0< \lambda _i < 1 \text{ 
for all } i\in I\right\}.
\]
We set $b_{\varphi}(I):=\varphi (\Box (I) )$ and observe that if $\varphi$ is 
combinatorially positive then, from Theorem~\ref{thm:combpositive}, we obtain 
$b_{\varphi}(I)\geq 0$ for all $I\subseteq [r]$.

Further, for each $I\subseteq [r]$ we define the half-open 
parallelepipeds
\[
\halflozenge (I) \ := \ \left\{\sum _{i=1}^r \lambda _i v_i : \,  0\leq \lambda _i < 
1 \text{ for all } i\in I,0\leq  \lambda _i \leq 1 \text{ for all } i\not\in 
I\right\}
\]
and
\[
\Pi (I) \ := \ \left\{\sum _{i\in I} \lambda _i v_i : \,  0\leq \lambda _i < 1 \text{ 
for all } i\in I\right\}.
\]
Note that $\lozenge ([r])=\halflozenge (\emptyset)$; we also set $\lozenge 
(\emptyset)=\Pi (\emptyset)=\{0\}$. 
The following lemma of Schepers and Van Langenhoven 
\cite{svl} was originally stated only for discrete volumes.
However, their proof works as well for arbitrary $\mathbb{Z}^d$-valuations.
\begin{lem}[{\cite[Lemma~2.1]{svl}}]\label{lem:Ehrhartfunction}
    Let $\varphi$ be a $\mathbb{Z}^d$-valuation and let $I\subseteq [r]$. Then
    \[
    \varphi (n \, \halflozenge (I)) \ = \ \sum _{J \supseteq I} n^{\left|J\right|} \varphi ( \Pi ({J}) ) \, .
    \]
\end{lem}
\begin{proof}
    To keep this paper self contained we give a  proof 
    here (slightly modified from that of~\cite{svl}). 
    As $v_1,\ldots,v_{r}$ are linearly independent, for every $x \in 
    \lozenge ([r]) $ there are unique $\lambda _1 , \ldots, \lambda _{r}\in 
    [0,1]$ such that
    \[
    x \ = \ \sum _{i=1} ^{r} \lambda _i v_i \, .
    \]
    Let $J_x := \{i\in [r] : \,  \lambda _i <1\}$. Then $x\in \Pi (J)+\sum _{i\not \in J}v_i$ 
    if and only if $J_x=J$.
    We observe that $x \in \halflozenge(I)$ if and only if $I \subseteq J_x$,
    and therefore we can  partition
    $$\halflozenge (I) \ = \ \biguplus _{J \supseteq I} \left(\Pi (J)+\sum _{i\not \in J}v_i\right) \, .$$
    Further, for all $J\subseteq [r]$ and all $n\geq 1$ we can tile $n \, \Pi (J)$ 
    with $n^{\left|J\right|}$ translates of $\Pi (J)$. Thus by the translation-invariance of $\varphi$,
    \[
    \varphi ( n \, \halflozenge (I))
    \ = \ \sum_{J \supseteq I} \varphi ( n \, \Pi (J))
    \ = \ \sum_{J \supseteq I} n^{\left|J\right|}\varphi ( \Pi (J)) \, . \qedhere
    \]
\end{proof}
Applying Lemma~\ref{lem:Ehrhartfunction} to the linearly independent standard 
basis vectors $e_1,\ldots, e_d$ and the discrete volume, we obtain the 
following corollary:
\begin{cor}\label{cor:Ehrhartfunctioncube}
    Let $0\leq j\leq d$. Then the Ehrhart polynomial of the half-open unit cube 
    $C_j ^d$ equals
    \[
    \ehr_{C_j ^d}(n)
    \ = \ \sum _{J \supseteq [j]} n^{\left| J \right|}
    \ = \ n^j (1+n)^{d-j},
    \]
    where we set $[0]:=\emptyset$.
\end{cor}
Corollary~\ref{cor:Ehrhartfunctioncube} together with Theorems~\ref{thm:jEulerianrealrooted} and~\ref{thm:hvectorunitcube} reproves the following result by Br\"and\'en~\cite{branden2006linear}.
\begin{cor}[\cite{branden2006linear}]\label{cor:branden}
Let $c_0,\ldots,c_d \geq 0$ be real numbers and let $h(t)\in \mathbb{R}[t]$ be the unique polynomial such that
\[
\sum _{n\geq 0}\sum _{j=0}^d c_j \, n^j (1+n)^{d-j} t^n\ = \ \frac{h(t)}{(1-t)^{d+1}} \, .
\]
Then $h(t)$ is real rooted.
\end{cor}

Recall our notation $b_{\varphi}(I)=\varphi (\Box (I) )$.

\begin{lem}\label{lem:boxpartition}
    Let $\varphi$ be a $\mathbb{Z}^d$-valuation. Then for all $I\subseteq [r]$,
    \[
    \varphi ( \Pi (I)) \ = \ \sum _{J\subseteq I} b_{\varphi}(J) \, .
    \]
\end{lem}
\begin{proof}
    For $x\in \Pi (I)$ there are unique $\lambda _i\in [0,1)$ such that
    \[
    x \ = \ \sum _{i\in I} \lambda _i v_i \, .
    \]
    Let $J_x := \{ i \in I : \,  \lambda _i =0 \}\subseteq I$. For all 
    $J\subseteq I$ we have $J=J_x$ if and only if $x\in \Box (I\setminus J)$. Therefore 
    \[
    \Pi (I) \ = \ \biguplus _{J\subseteq I}\Box (I\setminus J) ,
    \]
    and the result follows by the translation-invariance of $\varphi$.
\end{proof}
The following theorem generalizes \cite[Proposition~2.2]{svl}.
\begin{thm}\label{thm:halfopenparallelepipeds}
    Let $\varphi$ be a $\mathbb{Z}^d$-valuation. Then for all $I\subseteq [r]$,
    \[
    \varphi (\halflozenge (I),t) \ = \ \frac{\sum _{K\subseteq [r]} 
    b_{\varphi}(K) A _{\left|I\cup K\right|+1}(r+1,t)}{(1-t)^{r+1}} \, .
    \]
\end{thm}
\begin{proof}
    We follow the line of argumentation in \cite[Proposition~2.2]{svl}.  By Lemmas~\ref{lem:Ehrhartfunction} and~\ref{lem:boxpartition},
    \begin{align*}
        \varphi(\halflozenge (I),t)
        \ &= \ \sum_{n=0}^{\infty} t^n\sum_{J \supseteq I} n^{\left| J\right|} \, \varphi ( \Pi (J))\\
          &= \ \sum_{n=0}^{\infty}t^n \sum_{J \supseteq I} n^{\left| J\right|}\sum _{K \subseteq J} b_{\varphi}(K)\\
          &= \ \sum_{K\subseteq [r]}b_{\varphi}(K)\sum _{n=0}^{\infty}t^n\sum_{J \supseteq I\cup K} n^{\left| J\right|}.
    \end{align*}
    By Corollary~\ref{cor:Ehrhartfunctioncube},
    \[
    \sum_{J \supseteq I\cup K} n^{\left| J\right|} \ = \ \ehr_{C_{\left|I \cup K\right|}^{r}}(n) \, .
    \]
    The claim now follows from Corollary~\ref{cor:Ehrhartfunctioncube}.
\end{proof}
As a corollary we obtain unimodality of the $h^{\ast}$-vectors of half-open 
parallelepipeds in the case that $\phi$ is combinatorially positive.
\begin{cor}\label{cor:unimodhalfopen}
    Let $\varphi$ be a combinatorially positive $\mathbb{Z}^d$-valuation and 
    let 
    $I\subseteq [r]$. Then $h^{\varphi}(\halflozenge 
    (I))(t)$ has only real roots. Moreover, if $h^{\varphi}(\halflozenge 
    (I))=(h_{0},\ldots,h_{r},0,\ldots,0)$ is 
    the $h^{\ast}$-vector of $\halflozenge (I)$, then 
    \[
    h_{0}\leq \dots \leq h _{\frac{r}{2}} \geq 
    \dots \geq h _{r} \quad \text{ if } r \text{ is even}
    \]
    and
    \[
    h _{0}\leq \dots \leq h _{\frac{r-1}{2}} 
    \quad \text{ and } \quad h _{\frac{r+1}{2}} \geq \dots \geq 
    h _{r} \quad \text{ if } r \text{ is odd.}
    \]
\end{cor}

\begin{proof}
    By Theorem~\ref{thm:halfopenparallelepipeds},
    \[
    h^{\phi}(\halflozenge (I))(t) \ = \ \sum _{K\subseteq [r]} b_{\varphi}(K) 
    A _{\left|I\cup K\right|+1}(r+1,t) \, .
    \]
    As $\varphi$ is combinatorially positive, $b_{\varphi}(K)\geq 0$ for 
    all $K\subseteq [r]$ by Theorem~\ref{thm:combpositive}. By Theorem~\ref{thm:jEulerianrealrooted} 
    $h^{\phi}(\halflozenge (I))(t)$ is real-rooted. Moreover, by Theorem~\ref{thm:jEulerian_unimodality}, the coefficients of 
    $A _{\left|I\cup K\right|+1}(r+1,t)$ form a unimodal sequence with peak at 
    $\lfloor\frac{r+1}{2}\rfloor =\frac{r}{2}$ if $r$ is even, and peak at 
    $\frac{r-1}{2}$ or $\frac{r+1}{2}$ if $r$ is odd, and so 
    does any nonnegative linear combination.
\end{proof}


\subsection{Zonotopes}\label{sec:zonotopes}
    The following well-known result is due to Shephard \cite{shephardzonotopes}.
    \begin{thm}[{\cite[Theorem 54]{shephardzonotopes}}]\label{thm:Shephard}
        Every (lattice) zonotope has a subdivision into (lattice) 
        parallelepipeds.
    \end{thm}
    \begin{prop}\label{lem:partition}
        Let $\zon$ be an $d$-dimensional zonotope. Then 
        $\zon$ can be partitioned into $d$-dimensional half-open 
        parallelepipeds in the sense of Section 
        \ref{sec:halfopenparallelepiped}.
    \end{prop}
    \begin{proof}
                    Let $\zon=P_1\cup \cdots \cup P_k$ be a dissection of $\zon$ 
                into parallelepipeds, which exists by Theorem~                \ref{thm:Shephard}, and let $q\in \zon$ be a generic point. 
                Then, by Lemma \ref{lem:ho}, $\zon =  \half{q}{P_1} \uplus 
                \cdots \uplus \half{q}{P_k}$ is a partition into half-open 
                polytopes. Moreover, every $\half{q}{P_i}$ equals, up to 
                translation, some $\halflozenge (I)$: Let $F_1,\ldots, F_m$ be 
                the facets of $P_i$ visible from $q$. At most one of two 
                parallel facets can be visible; in particular, $0\leq 
                m\leq d$, and there is at least one vertex of $P_i$ that is not visible from $q$. Let $w$ be such a vertex. As 
                $P_i$ is a simple polytope, the vertex figure at $w$ is a 
                simplex, i.e., every facet 
                containing $w$ is uniquely determined by the neighbor vertex of 
                $w$ it 
                does not contain. For $1\leq i \leq m$, let $w_i$ be the 
                neighbor of $w$ 
                that is not contained in the facet parallel to $F_i$ and let $w_{m+1},\ldots , w_d$ 
                be the 
                other neighbors of $w$. Let $v_i:=w_i-w$ for all $1\leq i \leq 
                d$. Then with $I=[m]$ we obtain
                \[
                \half{q}{P_i} \ = \ P_i \setminus \bigcup _{i=1}^m F_i \ = \ 
                \halflozenge (I) +w \, . \qedhere
                \]
    \end{proof}
    As a consequence we deduce one of our main theorems.
    \begin{proof}[Proof of Theorem~\ref{thm:unimodalzonotope}]
        By Proposition~\ref{lem:partition} there is a partition $\zon =  
        \half{q}{P_1} \uplus 
        \cdots \uplus \half{q}{P_k}$ into half-open zonotopes, where each 
        $\half{q}{P_i}$ 
        equals, up to translation, some $\halflozenge (I)$. By 
        Corollary~\ref{cor:unimodhalfopen} and 
        Theorem~\ref{thm:jEulerianrealrooted},
        $h^{\phi} (\zon)(t)=\sum_{i}h^{\phi} (\half{q}{P_i})(t)$ is 
        real-rooted. Moreover, every 
        $\half{q}{P_i}$ has a unimodal $h^{\ast}$-vector with 
        peak at $ \frac{d}{2}$ if $d$ is even and peak at $\frac{d-1}{2}$ or at 
        $\frac{d+1}{2}$ if $d$ is odd. The same is true for $h^{\phi}  (\zon)$ 
        since it is a nonnegative linear combination of 
        $h^{\phi}(\half{q}{P_1}),\ldots,h^{\phi} (\half{q}{P_k})$.
        \begin{figure}[h]
            \centering
            \includegraphics[width=10cm]{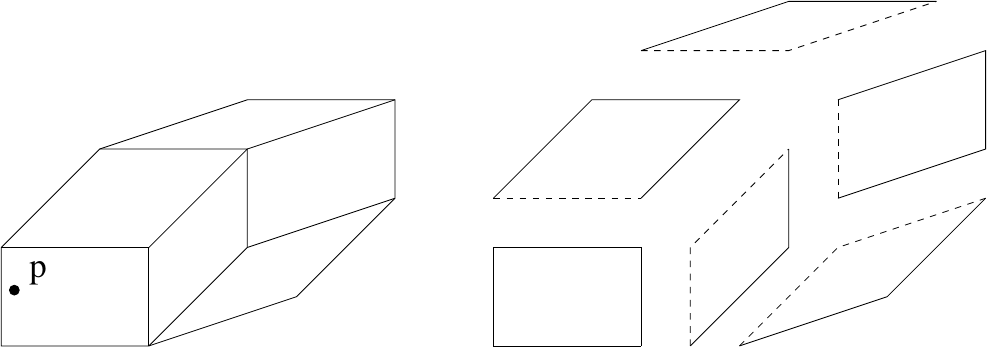}
            \caption{A zonotope and one of its half-open decomposition into 
            parallelepipeds.}
        \end{figure}
    \end{proof}	


\subsection{Ehrhart $h^\ast$-vectors of zonotopes}
Theorem~\ref{thm:halfopenparallelepipeds} allows us to explicitly describe the convex hull of all Ehrhart $h^\ast$-vectors of zonotopes. Let
\[
\mathcal{Z}_d \ := \ \conv \left\{ h^\ehrval (P)  : \,  P \text{ $d$-dimensional 
lattice zonotope} \right\} 
\]
and 
\[
\mathcal{P}_d \ := \ \conv \left\{ h^\ehrval (P)  : \,  P \text{ $d$-dimensional 
lattice parallelepiped} \right\} .
\]
For all $1\leq j \leq d$ set $A_j(d) = (A_j(d,0),A_j(d,1),\ldots, A_j(d,d))$.
Theorem~\ref{thm:hconvexhull} says that
\[
\mathcal{Z}_d \ = \ \mathcal{P}_d \ = \ A_1(d+1) + \mathbb{R}_{\geq 0} \, A_2 
(d+1) + \cdots + \mathbb{R}_{\geq 0} \, A_{d+1} (d+1) \, ,
\]
and that this cone is simplicial.

\begin{proof}[Proof of Theorem~\ref{thm:hconvexhull}]
From the proof of Theorem~\ref{thm:unimodalzonotope} we see that the 
$h^\ast$-polynomial of every zonotope is a nonnegative linear combination of 
$(A,j)$-Eulerian
polynomials. Moreover, since in every half-open decomposition used in the 
proof, there is at most one closed parallelepiped, we see from 
Theorem~\ref{thm:halfopenparallelepipeds} that the multiplicity of $A_1(d+1)$ 
equals $b(\emptyset )=1$. Therefore, $\mathcal{Z}_d$ and $\mathcal{P}_d$ are 
contained in $A_1(d+1) 
+ \mathbb{R}_{\geq 0} \, A_2 
(d+1) + \cdots + \mathbb{R}_{\geq 0} \, A_{d+1} (d+1)$.

For the reverse inclusions, since $\mathcal{P}_d \subseteq \mathcal{Z}_d$, it 
suffices to prove 
that for every integer $m$ and every $2\leq k 
\leq d+1$ there is a parallelepiped $P_{k,m}$ with $h^\ehrval 
(P_{k,m})=A_1(d+1)+m \, A_k(d+1)$. Consider the parallelepiped
\[
  P_{k,m} \ = \ \left\{ \sum _{i=1}^d \lambda _i \, v_i  : \,  0\leq \lambda _i \leq 1 \right\}
\]
with $v_k=e_1+\dots +e_{k-1}+(m+1)e_k$ and $v_i=e_i$ for all other $i$. From 
the proof of \cite[Theorem 2.2]{stanley} it follows that $|\Pi 
(I)\cap \mathbb{Z}^d|$ equals the absolute value of the maximal minor of the 
matrix with columns $\{v_i\}_{i\in I}$. It is therefore not hard to calculate 
that $|\Pi (I)\cap \mathbb{Z}^d|=m+1$ if and only if $[k]\subseteq I$, and $|\Pi 
(I)\cap \mathbb{Z}^d|=1$ otherwise. Thus, by Lemma~\ref{lem:boxpartition} and 
M\"obius inversion,
\[
b_\ehrval (I) \ = \ \begin{cases}
m & \text{ if } I=[k],\\
1     & \text{ if } I=\emptyset,\\
0     & \text{otherwise.}
\end{cases}
\]
Theorem~\ref{thm:halfopenparallelepipeds} now gives 
$P_{k,m}=A_1(d+1)+m \, A_k(d+1)$. 

Moreover, the vectors $A_1(d+1),\ldots, 
A_{d+1}(d+1)$ are linearly independent:
by Corollary~\ref{cor:Ehrhartfunctioncube}, $\ehr_{C_j ^d}(n) 
= n^j(1+n)^{d-j}$, and the polynomials 
$\{n^j(1+n)^{d-j}\}_{j=0,\ldots,d}$ form a basis of the space of polynomials of 
degree at most $d$. Considering instead the $h^\ast$-polynomial defines a basis 
transformation, and $h^{\ehrval}(C_j ^d)(t)=A_{j+1}(d+1,t)$ by Theorem 
\ref{thm:hvectorunitcube}, so the polynomials $\{A_{j+1}(d+1,t)\}_{j\in [d]}$ 
also define a basis.
\end{proof}

 Theorem \ref{thm:hconvexhull} naturally gives rise to the following open problem.
 \begin{problem}\label{prob:1}
 Characterize the sets of all $h^\ast$-vectors of $d$-dimensional parallelepipeds/zonotopes.
 \end{problem}
 
 We suspect that this problem is quite nontrivial.
 From our proof of Theorem~\ref{thm:unimodalzonotope} we see that every such 
 $h^\ast$-vector is contained in $A_1(d+1) + \mathbb{Z}_{\geq 
 0} \, A_2 (d+1) + \cdots + \mathbb{Z}_{\geq 0} \, A_{d+1} (d+1)$. However, it 
 is easy to check that already for $d=2$,
 \[
  A_1(d+1) + \mathbb{Z}_{\geq 0} \, A_2 (d+1) + \cdots + \mathbb{Z}_{\geq 0} \, 
  A_{d+1} (d+1) \subsetneq \mathcal{Z}_d\cap \mathbb{Z}^d \, ,
  \]
  and so $\mathcal{Z}_d\cap \mathbb{Z}^d$ cannot be the right answer. But 
  $A_1(d+1) + \mathbb{Z}_{\geq 0} A_2 (d+1) + \cdots + \mathbb{Z}_{\geq 0} 
  A_{d+1} (d+1)$ does not characterize all $h^\ast$-vectors of $d$-zonotopes either, since 
  with Lemma~\ref{lem:jEulerian_symmetry} it is not hard to see that this affine 
  semigroup contains 
  infinitely many symmetric vectors, i.e., vectors of the form 
  $(a_0,a_1,\ldots,a_d)$ with $a_i=a_{d-i}$ for all $0\leq i\leq d$. By a 
  theorem of Hibi \cite{hibi1992note}, the corresponding polytopes are reflexive; 
  however, there are only finitely many reflexive polytopes in each dimension, by a 
  result of Lagarias and Ziegler~\cite{LagariasZiegler}. Although a complete 
  solution of Problem~\ref{prob:1} might be out of reach, it would be 
  instructive to see whether the set of $h^\ast$-vectors exhibits interesting 
  combinatorial or algebraic structures.

From Theorem~\ref{thm:hvectorunitcube} and \cite{hibi1992note} we deduce the following characterization of reflexive polytopes.
\begin{prop}\label{prop:reflexive}
Let $P$ be a $d$-dimensional lattice polytope and let 
\[
\ehr_{P}(n) 
= \sum _{j=0}^dc_jn^j(1+n)^{d-j}
\]
be its Ehrhart polynomial. Then $P$ is reflexive if and only if $c_j=c_{d-j}$.
\end{prop}
\begin{proof}
By Theorem \ref{thm:hvectorunitcube} and Corollary  \ref{cor:Ehrhartfunctioncube} the $h^\ast$-polynomial of $P$ equals $\sum _{j=0}^dc_jA_{j+1}(d+1,t)$. By  Lemma~\ref{lem:jEulerian_symmetry} this polynomial is palindromic if and only if $c_j=c_{d-j}$ for all $0\leq j\leq d$. On the other hand, by Hibi's result \cite{hibi1992note} the $h^\ast$-polynomial is palindromic if and only if $P$ is reflexive.
\end{proof}

\section{Alternatingly Increasing $h^\ast$-vectors}\label{sec:altincr}
In \cite{svl} Schepers and Van Langenhoven conjectured that the Ehrhart $h^\ast$-vector of every polytope having the IDP property and an interior lattice point is alternatingly increasing.
\begin{conj}[{\cite{svl}}]
Let $P$ be an $r$-dimensional polytope having the IDP property with a lattice point in its 
relative interior. Then
\[
h^\ehrval _0 (P)\leq h^\ehrval _r (P)\leq h^\ehrval _1 (P)\leq  h^\ehrval _{r-1} 
(P)\cdots \leq h^\ehrval _{\lfloor\frac{r+1}{2}\rfloor (P)} \, .
\]
\end{conj}\label{conj:alternatinglyincr}
Schepers and Van Langenhoven proved their conjecture for lattice parallelepipeds with an interior lattice point. 
Towards that conjecture, we prove that the $h^\ast$-vector with respect to any 
combinatorially positive valuation of any lattice centrally symmetric  or coloop-free 
zonotope is alternatingly increasing.


\subsection{Type-$B$ zonotopes}\label{sec:typeBzonotopes}
A \textbf{type-$B$ zonotope} (or \textbf{lattice centrally symmetric 
zonotope}) in $\mathbb{R}^d$ is a lattice zonotope that arises as a 
projection of $[-1,1]^m$ for some $m\geq d$. Equivalently, a zonotope 
$\zon$ is lattice centrally symmetric if there are vectors $v_1,\ldots, 
v_m \in \mathbb{Z}^d$ such that
\[
\zon \ = \ \left\{\sum _{i=1}^m \lambda_i v_i  : \,  -1\leq \lambda _i \leq 
1\right\} 
\]
up to translation by an element in $\mathbb{Z}^d$. A central role in 
determining the $h^\ast$-polynomial of lattice centrally symmetric 
zonotopes is played by the type-$B$ Eulerian polynomials.
To that end we define the half-open $\pm1$-cube
$[-1,1]^d_{l}$, where $d \ge 1$ and $0 \leq l \leq d$, by
\begin{align*}
[-1,1]_{l}^d 
 \ &:= \ [-1,1]^d \setminus \{ x_d = 1,  x_ {d-1} =1, \dots,  x_{d+1-l}=1 \} \, 
 ,
\end{align*}
and the {\bf $l$-descent set}, 
$\Des_{l}(\piep) \subseteq \{0, 1, \dots, d \}$, of signed permutation 
$(\piep) \in B_d$ by
\[
\Des _{l} (\piep) \ := \
\begin{cases}
\Des (\piep) \cup \left\{ d \right\}   & \text{ if } d+1-l \le 
\epsilon_{d}\pi_{d} \le d \, , \\
\Des (\piep)  & \text{ otherwise.} \\
\end{cases}
\]
The cardinality of this set is the {\bf (natural) $l$-descent} number of 
$(\piep)$, denoted by
\[
\des_{l}(\piep) \ := \ \left| \Des_{l}(\piep) \right| .
\]

\begin{thm}\label{thm:h_typeB}
\label{thm:Ehrhart_B_ell}
\[
	\Ehr\left( [-1,1]^d_{l}, t \right) \ = \ \frac{B_{l+1}(d+1,t)}{(1-t)^{d+1}} 
	\, .
\]
\end{thm}
To prove this theorem we will use the following lemma which parallels Lemma~\ref{lem:jdescents}.
\begin{lem}
\label{lem:bijection_descents}
Let $d \ge 1$ and $0 \le l \le d$. Then
\[ 
	b_{l+1}(d+1, k) \ = \ \left| \left\{ (\piep) \in B_d : \, \des_{l}(\piep) = k 
	\right\} \right| .
\]
\end{lem}
\begin{proof}[Proof of Theorem~\ref{thm:Ehrhart_B_ell}]
We use a half-open decomposition of $[-1,1]^d$ induced by the hyperplanes given by $x_i=\pm x_j$ and $x_i=0$, for $0\leq i < j\leq d$. As before, we obtain a decomposition into half-open unimodular simplices

\begin{align*}
\triangle_{(\piep)}^{d, l} \ &= \
	\begin{Bmatrix}
		\x \in [-1, 1]^d_{l} : \,
		0 \le \epsilon_{1} x_{\pi_{1}} \le \dots \le \epsilon_{d} x_{\pi_{d}} \le 1 \\[0.3em]
		\text{with } \epsilon_{i} x_{\pi_i} < \epsilon_{i+1} x_{\pi_{i+1}} \text{ when } i \in \Des(\piep) \\[0.3em]
	\end{Bmatrix}\\
   &= \ \begin{Bmatrix}
		\x \in \R^{d} : \,
		0 \le \epsilon_{1} x_{\pi_{1}} \le \dots \le \epsilon_{d} x_{\pi_{d}} \le 1 \\[0.3em]
		\text{with } \epsilon_{i} x_{\pi_i} < \epsilon_{i+1} x_{\pi_{i+1}} 
		\text{ when } i \in \Des_{l}(\piep) \\[0.3em]
		\text{and } \epsilon_{d} x_{\pi_d} < 1 \text{ when } d \in 
		\Des_{l}(\piep) \\[0.3em]
	\end{Bmatrix} .
\end{align*}
The number of missing facets of $\triangle_{(\piep)}^{d, l}$ is therefore 
$\des_{l}(\piep)$.
Observe that
\[
	[-1,1]_{l}^d \ = \biguplus_{(\piep) \in B_d} \triangle_{(\piep)}^{d, l}
\]
as a disjoint union.
Therefore, 
\[
\Ehr \left( [-1,1]_{l}^d, t \right)
\ = \ \sum_{\sigma \in S_d} \Ehr \left( \triangle_{(\piep)}^{d, l}, t \right) \ 
= \ \frac{ \sum_{\sigma \in S_d} t^{\des_{l}(\piep)} }{(1-t)^{d+1}} \ = \ 
\frac{B_{l+1}(d+1, t)}{(1-t)^{d+1}} \, ,
\]
where the last equality holds by Lemma~\ref{lem:bijection_descents}.
\end{proof}
\begin{proof}[Proof of Proposition~\ref{thm:typeB}]
By Theorem~\ref{thm:h_typeB}, the Ehrhart $h^\ast$-polynomial of $[-1,1]_l^d$ 
equals $B_{l+1}(d+1,t)$. Theorem~\ref{thm:halfopenparallelepipeds} allows us 
to 
give an expression in terms of $(A,j)$-Eulerian polynomials: observe that 
$[-1,1]_l^d$ is lattice 
isomorphic to $\halflozenge ([l])$ for $v_1=2e_1, v_2=2e_2,\ldots, 
v_d=2e_d$. 
Since $| \Box (K) \cap \mathbb{Z}^d|=1$ for 
all $K\subseteq [d]$ and by simple counting,
\[
  B_{l+1}(d+1,t)
  \ = \ \sum _{K\subseteq [d]} \left| \Box (K) \cap \mathbb{Z}^d \right| 
  A_{|[l]\cup K|+1}(d+1,t)
  \ = \ 2^l\sum _{j=0}^{d-l}{d-l\choose j}A_{j+l+1}(d+1,t) \, . \qedhere
\]
\end{proof}

Up to lattice translation, every half-open type-$B$ parallelepiped equals  
$2 \, \halflozenge (I)$ for some $v_1,\ldots, v_d\in \mathbb{Z}^d$.
\begin{thm}
    \label{thm:Ehrhart_lcs}
Let $\phi$ be a $\mathbb{Z}^d$-valuation. Then
    \[
    \Ehr^\phi \left( 2 \, \halflozenge (I),t\right) 
    \ = \ \frac{ \sum_{K \subseteq [d]} b_\phi(K) 
        \, B_{|I \cup K| +1}(d+1, t) }{(1-t)^{d+1}} \, .
    \]
\end{thm}
\begin{proof}
Since $[-1, 1]^d_{l}$ is lattice isomorphic to $2\,C^d_{l}$, by Corollary 
\ref{cor:Ehrhartfunctioncube}
\[
\ehr_{ [-1,1]^d_{l} } (n) \ = \ \sum_{ [l] \subseteq J 
\subseteq [d] } (2n)^{|J|} \, .
\]
By Lemmas~\ref{lem:Ehrhartfunction} and~\ref{lem:boxpartition},
    \begin{align*}
        \Ehr^\varphi \left( 2 \, \halflozenge (I),t \right)\ &= \ \sum _{n=0}^{\infty} t^n\sum 
        _{J \supseteq I} (2n)^{\left| J\right|} \varphi ( \Pi (J))\\
        &= \ \sum _{n=0}^{\infty}t^n \sum _{J \supseteq I}(2n)^{\left| 
            J\right|}\sum _{K \subseteq J} b_{\varphi}(K)\\
        &= \ \sum _{K\subseteq [r]}b_{\varphi}(K)\sum _{n=0}^{\infty}t^n\sum 
        _{J \supseteq I\cup K}(2n)^{\left| J\right|} \, ,
    \end{align*}
    and so the claim follows from Theorem~\ref{thm:h_typeB}.
\end{proof}
\begin{cor}
\label{cor:Ehrhart_lcs}
	The $h^\ast$-polynomial with respect to any combinatorially positive 
	valuation of a half-open type-$B$ parallelepiped has alternatingly 
	increasing coefficients.
\end{cor}
\begin{proof}
Since $\phi$ is combinatorially positive, by Theorems~\ref{thm:combpositive} and~\ref{thm:Ehrhart_lcs}, the 
$h^\ast$-polynomial is a nonnegative linear combination of $(B,l)$-Eulerian 
polynomials. These polynomials are alternatingly increasing (Theorem~\ref{thm:typeBalternatinglyincr}), and so is every nonnegative linear 
combination.
\end{proof}
\begin{cor}\label{cor:typeBparallelepipedalternatinglyincr}
The $h^\ast$-polynomial of a type-$B$ zonotope has 
alternatingly increasing coefficients.
\end{cor}
\begin{proof}
From Theorem~\ref{thm:Shephard} and Corollary~\ref{lem:partition} we 
see that every type-$B$ zonotope can be partitioned into half-open 
type-$B$ parallelepipeds, which are up to translation all of the form 
$2 \, \halflozenge (I)$ for some spanning vectors. By Corollary 
\ref{cor:typeBparallelepipedalternatinglyincr} every such half-open type-$B$ 
parallelepiped has an alternatingly increasing $h^\ast$-polynomial, and so does 
their sum which equals the $h^\ast$-polynomial of the zonotope.
\end{proof}

\subsection{Matroidal aspects of $h^\ast$-polynomials}\label{sec:matroidalasp}
The terminology used in this section originates from matroid theory;
see, e.g., \cite{oxley2006matroid}. Let  
$V=\{v_1,\ldots, v_n\} \subset \mathbb{Z}^d$ be a set of 
vectors, and $v_1<\cdots < v_n$ be a fixed order on the elements of $V$. Let 
$\mathcal{I}$ denote the set of independent subsets of $V$ and $\mathcal{B}$ be 
the collection of maximally independent subsets (i.e., bases). Without loss of generality, in 
the sequel we assume that $V$ spans $\mathbb{R}^d$. For simplicity, we identify 
$v_i$ with $i\in [n]$ in the sequel. The order on 
$[n]$ induces an order on $\mathcal{B}$, namely 
$B_1<B_2$ whenever $B_1$ is 
lexicographically smaller than $B_2$. For 
every independent set $I\in 
\mathcal{I}$, we define $\lfloor I \rfloor := \min _{B\in 
\mathcal{B}}\{I\subseteq B\}$, i.e., $\lfloor I \rfloor$ is the smallest basis 
that contains $I$. An element $i$ in a basis $B$ is called {\bf internally passive}
if there is an element $j<i$ with $j\notin B$, such that $\{j\}\cup B\setminus 
\{i\}$ is a basis. In other words, $i$ can be exchanged with a smaller 
element $j$. We denote the set of internally passive elements of $\mathcal{B}$ 
by $\IP (B)$. Note that, in particular, $\IP (B) \subseteq B$.

\begin{lem}\label{lem:shelling}
    Let $I\in \mathcal{I}$ and $B\in \mathcal{B}$. If $\lfloor I \rfloor \neq 
    B$ then there is an $i\in \IP (B)$ such that $I\subseteq B\setminus \{i\}$.
\end{lem}

\begin{proof}
    Without loss of generality we may assume that $\lfloor I \rfloor \cap B=I$. 
    Since $\lfloor I \rfloor \neq B$, there is $j\in \lfloor I \rfloor 
    \setminus I$ that is smaller than all elements in $B\setminus I$. Since 
    $j\cup I\in \mathcal{I}$, by Steinitz's Exchange Lemma, there exists $i\in 
    B\setminus I$ such that $j\cup B\setminus \{i\}$ is a basis, in particular, 
    $i\in \IP (B)$ and moreover $I\subseteq B\setminus \{i\}$.
\end{proof}

\begin{lem}\label{lem:Imax}
    Let $I\in \mathcal{I}$ and $B\in \mathcal{B}$ such that $I\subseteq B$. Then
    \[
    \lfloor I \rfloor =B \quad \Longleftrightarrow \quad \IP (B) \subseteq I \, .
    \]
\end{lem}
\begin{proof}
    For the forward direction, suppose that there is an 
    $i\in \IP (B)$ that is not contained in $I$. Since $i\in \IP (B)$ we obtain 
    $\lfloor B\setminus \{i\}\rfloor 
    \not = B$. Because $I\subseteq B\setminus 
    \{i\}$ it also follows that $\lfloor 
    I\rfloor \not = B$, which is a contradiction.
    
    For the backward direction, 
    it suffices to prove $\lfloor \IP (B) \rfloor =B$. Suppose 
    $\lfloor \IP (B) \rfloor \not =B$, then by Lemma~\ref{lem:shelling} there 
    exists an $i\in \IP (B)$ with $ \IP (B)  \subseteq B\setminus\{i\}$, a 
    contradiction.
\end{proof}
In the sequel we freely make use of the notation from Section~\ref{sec:halfopenparallelepiped}. The following lemma by Stanley~\cite{stanley} 
(with a proof by Ziegler) was used to prove Theorem \ref{thm:stanleyEhrhart}. (Compare also~\cite{betke1986application}.)
\begin{lem}[{\cite[Lemma 2.1]{stanley}}]\label{lem:decompstanley}
Let $Z$ be the zonotope generated by $V$. Then
\[
    Z \ = \ \biguplus _{I\in \mathcal{I}} \Pi(I)
\]
up to translation of each half-open parallelepiped on the right-hand side.
\end{lem}

With Lemma \ref{lem:decompstanley}, \cite[Theorem 2.2]{stanley} extends 
immediately to calculating 
$\ehr^\phi_Z (n)$ for arbitrary $\Z^d$-valuations. We formally 
record 
this with the following proposition.
\begin{prop}
Let $Z$ be the zonotope generated by $V$, and let $\varphi$ be a 
$\Z^d$-valuation. Then
\[
\ehr^\phi_P (n) \ = \ \sum _{I\in \mathcal{I}} \phi (\Pi (I)) \, n^{|I|} .
\]
\end{prop}

We are now able to prove Theorem \ref{thm:zonotopevaluation}.

\begin{proof}[Proof of Theorem \ref{thm:zonotopevaluation}]
    By Lemmas~\ref{lem:Ehrhartfunction} and~\ref{lem:Imax}, up to translation of 
    the half-open parallelepipeds, 
    \[
    Z \ = \ \biguplus _{B\in \mathcal{B}} \biguplus 
        _{I:\lfloor I 
        \rfloor =B} \Pi(I) \ = \ \biguplus _{B\in \mathcal{B}}\halflozenge 
        (\IP(B)) \, .
    \]
    The claim now follows from Theorem~\ref{thm:halfopenparallelepipeds}.
\end{proof} 
If the vectors in $V$ are the columns of a totally unimodular matrix, and $\phi 
=\ehrval$, the result of Theorem \ref{thm:zonotopevaluation} simplifies.
\begin{cor}
    Let $Z$ be the zonotope generated by the columns of a totally unimodular 
    matrix. Then 
    \[
    h^\ehrval (Z)(t)
    \ = \ \sum _{B\in \mathcal{B}}A_{|\IP 
        (B)|+1}(d+1,t) \, .
    \]
\end{cor}

A set of vectors $V=\{v_1,\ldots, v_n\} \subset \mathbb{R}^d$ is 
called {\bf coloop free} if there is no $v_i$ that is contained in every 
maximally independent subset of $V$. Equivalently, $V$ is coloop free if every 
linearly independent subset of $V$ that is not a basis is contained in at 
least two different bases formed by elements in~$V$.
\begin{thm}\label{thm:coloopfree}
Let $V\subset \mathbb{Z}^d$ be coloop free, let 
$Z$ be the zonotope generated by $V$ and let $\varphi$ be a combinatorially 
positive valuation. Then $h^\phi (Z)(t)$ is alternatingly increasing.
\end{thm}

\begin{proof}
    For all $B\in \mathcal{B}$ let $\bIP (B)$ denote the internally passive 
    elements of $B$ with respect to the reverse ordering $v_n<\ldots<v_1$.
    By Theorem~\ref{thm:halfopenparallelepipeds},
    \[
    h^\phi(Z)(t) \ = \ \frac{1}{2}\sum _{B\in \mathcal{I}}\sum_{I\subseteq 
        B}b_\phi(I)\left(A_{|I\cup \IP(B)|+1}(d+1,t)+A_{|I\cup 
        \bIP(B)|+1}(d+1,t)\right).
    \]
    Since $V$ is coloop free, for every $B\in \mathcal{B}$ and every $i\in B$ 
    there is $j\neq i$ such that $j\cup B\setminus i$ is a basis. Since 
    either $j<i$ or $j>i$ we obtain $i\in \IP (B) \cup \bIP (B)$. Therefore,   
    $|\IP (B)|+| 
    \bIP (B)|\geq |B|=d$ for every $B\in \mathcal{B}$, and hence
    \[
    |I\cup \IP (B)|+1+|I\cup \bIP (B)|+1 \ \ge \ d+2
    \]
    for all $I\subseteq B$. Since $\phi$ is combinatorially positive, $b_\phi 
    (I)\geq 0$ for all $I\subseteq B$ and therefore the claim follows from Lemma 
    \ref{lem:altincr}, since every nonnegative linear combination of 
    alternatingly increasing sequences is itself alternatingly increasing.
\end{proof} 

    \textbf{Acknowledgements:} We would like to thank Petter Br\"and\'en, Ben Braun, Felix Breuer, Aaron Dall, Martin Henk, Yvonne 
    Kemper, and Raman Sanyal for stimulating discussions and the anonymous referee for helpful comments. Matthias Beck was partially supported by the U.S.\ National Science Foundation (DMS-1162638).
    Katharina Jochemko was partially supported by a Hilda 
    Geiringer 
    Scholarship of the Berlin Mathematical School (BMS), and wants to thank the 
    BMS especially for funding her research stay at San Francisco State 
    University in Fall 2013, during which this project evolved. 
    Emily McCullough was partially supported by the U.S.\ National Science 
    Foundation (DGE-0841164).

\nocite{jochemko}
\nocite{mccullough}
\bibliographystyle{amsplain}
\bibliography{bib_jointpaper}


\end{document}